\documentclass[11pt,twoside]{amsart}
\usepackage[dvips]{graphicx}
\usepackage{amssymb}
\usepackage{latexsym}
\usepackage{xypic}

\newtheorem{theorem}{Theorem}[section]
\newtheorem{lemma}[theorem]{Lemma}
\newtheorem{proposition}[theorem]{Proposition}

\newtheorem{corollary}[theorem]{Corollary}
\newtheorem{Thm}{Theorem}

\theoremstyle{definition}
\newtheorem{definition}[theorem]{Definition}

\theoremstyle{remark}
\newtheorem{remark}[theorem]{Remark}
\newenvironment{definition-proposition}{\begin{def-prop} \em}{\end{def-prop}}

\def\P{\mathbb P_{\mathbb C}}
\def\H{\mathbb H}

\numberwithin{equation}{section}

\newcommand{\C}{\mathbb{C}}

\begin{document}

\title{  \sc{On the Equicontinuity Region of Discrete  Subgroups of} $PU(1,n)$}
\author{Angel Cano \& Jos\'e Seade}
\address{ Instituto de Matem\'aticas Unidad Cuernavaca, Universidad Nacional Autonoma de M\'exico, Avenida  Universidad sin n\'umero, Colonia Lomas de Chamilpa, Cuernavaca,  Morelos, M\'exico}
\email{angel@matcuer.unam.mx\\ jseade@matcuer.unam.mx}
\thanks{Research partially supported by grants from CONACYT   and PAPIIT-UNAM, Mexico, and by the Abdus Salam ICTP, Trieste, Italy.}




\begin{abstract}
 Let $ G  $ be a discrete subgroup of $PU(1,n)$. Then  $ G  $ acts on $\mathbb {P}^n_\mathbb C$
preserving the unit ball $\mathbb {H}^n_\mathbb {C}$, where it acts by isometries with respect to
the Bergman metric.    In this work we determine the equicontinuty region $Eq(G)$ of $G$ in $\mathbb  P^n_{\mathbb  C}$:  It is the complement of   the union of all complex projective
hyperplanes in $\mathbb {P}^n_{\mathbb C}$ which are tangent to $\partial \mathbb {H}^n_\mathbb {C}$
at  points in the Chen-Greenberg limit set $\Lambda_{CG}( G  )$, a closed $G$-invariant subset of $\partial \mathbb {H}^n_\mathbb {C}$, which is minimal for non-elementary groups. We also prove that the action on  $Eq(G)$ is discontinuous.

\end{abstract}

\maketitle

\section*{ Introduction}
 Let $PU(1,n)\subset PSL(n+1,\C)$ be the group of  automorphisms of
$\mathbb {P}^n_\mathbb {C}$ that preserve
 the ball
$$\{[z_0 : z_1: \ldots : z_{n}] \in \mathbb {P}^n_{\C} : |z_1|^2 + |z_2|^2+\ldots + |z_n|^2< |z_{0}|^2\}\,.$$
We   equip this ball with the Bergman metric, so we get a model for the complex hyperbolic space $\mathbb {H}^n_\mathbb {C}$, with $PU(1,n)$ as its group of holomorphic isometries. If $ G  \subset PU(1,n)$ is a discrete subgroup, then its limit set
$\Lambda_{CG}( G  )$  was defined by Chen-Greenberg in \cite{CG}
 as the set of accumulation points of the
$ G  $-orbits in $\mathbb {H}^n_{\mathbb C}$. As for conformal Kleinian
groups, this limit set  is contained in the ``sphere at
infinity'' $\partial \mathbb H^{n} _{\mathbb C}$, and it is a closed invariant set, which either has
cardinality $\leq 2$ or else it has infinitely many points, all orbits in it are dense and it is the unique minimal closed $G$-invariant set in $\mathbb {H}^n_\mathbb {C}$.
 The action of $G$ on $\mathbb H^{n} _{\mathbb C}$, being by isometries, is discontinuous and equicontinuous.

 We notice  that $G$ is  by definition  also a subgroup of $PSL(n+1, \mathbb C)$ so it acts on all of $\P^n$, and  it is natural to look for information
  about its  action  on all of $ \mathbb P^{n} _{\mathbb C}$, where the action is no longer isometric. This is analogous to considering a classical fuchsian group $\Gamma \subset PSL(2,\mathbb R)$ and thinking of it as acting in $\P^1$ preserving a ball, which serves as model for the real hyperbolic plane.

  In this article we
 interested in studying the region of equicontinuity in $\mathbb P^n_{\mathbb C}$ of discrete subgroups of $PU(1,n)$.   This is interesting, among other  reasons, because  equicontinuity is a gate for the analytic study of the dynamics of discrete subgroups of $PSL(n+1, \mathbb C)$.

 We prove:

\begin{Thm}\label{Thm:main}
Let $ G  \subset PU(1,n)$ be a discrete subgroup and let $Eq(G)$ be its equicontinuity region in $ \mathbb P^{n} _{\mathbb C}$. Then  $ \mathbb P^{n} _{\mathbb C} \setminus Eq( G  ) $   is the union
of all complex projective hyperplanes  tangent to $\partial \mathbb H^n
_{\mathbb C}$ at points in $\Lambda_{CG}( G  )$, and
 $ G  $ acts discontinuously on $Eq( G  )$. Furthermore,   the  set of accumulation points of the $G$-orbit  of every compact set $K\subset Eq( G  )$ is contained in $\Lambda_{CG}( G )$.

\end{Thm}

The proof of this theorem relies on   {\it quasi-projective} transformations, introduced by Furstenberg (see \cite{AMS,Furs}), which provide a completion of the non-compact Lie group $PSL(n+1,\C)$.

We ought to mention that this work was inspired by Navarrete's theorem in  \cite{pablo}, where the author
studies discrete subgroups of $PU(1,2)$ and compares the aforementioned Chen-Greenberg limit set with a different notion of limit set, due to R. Kulkarni \cite{kulkarni}, which has the property of granting that the action on its complement  $\Omega_{Kul}(G)$  is discontinuous. The theorem in \cite{pablo} says that the Kulkarni region of discontinuity $\Omega_{Kul}(G)$ is the complement of the union of all projective lines in $\P^2$ wich are tangent to $\partial
\H^2_\C$ at points in the Chen-Greenberg limit set. That proof  also shows implicitly   that for $n =2$ the regions $Eq( G  )$ and $\Omega_{Kul}(G)$ coincide: a fact  which is not known in higher dimensions.
Actually, for $n = 2$ an additional dimensional argument,
together with Theorem 1 give a simpler proof of the theorem in [8]; the
details are given in  \cite{CNS}.

We are grateful to Professor D. P. Sullivan for suggesting us to look at the equicontinuity set of  groups acting on $\P^n$.
Part of this research was done while the authors were visiting the ICTP at Trieste, Italy, and they are grateful to this institution and its people, for their support and hospitality.

\section{Preliminaries on projective and complex hyperbolic geometry}

We recall that the complex projective space $\mathbb {P}^n_{\mathbb {C}}$
is defined as:
$$ \mathbb {P}^{n}_{\mathbb {C}}=(\mathbb {C}^{n+1}- \{0\})/\C^* \,,$$
where the non-zero complex numbers are acting coordinate-wise.
  This is   a  compact connected  complex $n$-dimensional
manifold.

If $[\mbox{ }]_n:\mathbb {C}^{n+1}\setminus\{0\}\rightarrow
\mathbb {P}^{n}_{\mathbb {C}}$ is the quotient map, then a
non-empty set  $H\subset \mathbb {P}^n_{\mathbb {C}}$ is said to be a
projective subspace of dimension $k$  ({\it i.e.}, $dim_{\mathbb {C}}
(H)=k$) if there is a  $\mathbb {C}$-linear subspace  $\widetilde H$
of dimension $k+1$ such that $[\widetilde H]_n=H$. Hyperplanes are the  projective subspaces of dimension  $n-1$. Given distinct points
$p,q\in \mathbb{P}^n_{\mathbb{C}}$,   there is a
unique  complex projective subspace of dimension 1  passing through $p$ and
$q$. Such a subspace  will be called a complex (projective) line
and  denoted by $\overleftrightarrow{p,q}$; this is the image
under  $[\mbox{ }]_n$ of a two-dimensional linear subspace of
$\C^{n+1}$.

 It is clear  that every linear automorphism of $\C^{n+1}$ defines a
holomorphic automorphism of $\P^n$, and it is well-known that every automorphism of $\P^n$
arises in this way. Thus one has that the group of projective
automorphisms is:
$$PSL(n+1, \mathbb {C}) \,:=\, GL({n+1}, \C)/(\C^*)^{n+1} \cong SL({n+1}, \mathbb {C})/\mathbb {Z}_{n+1} \,,$$
where $(\C^*)^{n+1} $ is being regarded as the subgroup of
diagonal matrices with a single non-zero eigenvalue, and we
consider the action of $\mathbb {Z}_{n+1}$ (regarded as the roots of
unity) on  $SL(n+1, \mathbb {C})$ given by the usual scalar
multiplication. Then $PSL(n+1, \mathbb {C})$ is a Lie group whose
elements are called projective transformations.

We denote also by $[\mbox{  }]_{n}: SL(n+1, \mathbb {C})\rightarrow
PSL(n+1, \mathbb {C})$    the quotient map, which indeed is restriction of a map defined on the general linear group $GL(n+1,\mathbb  C)$. Given     $ \gamma  \in
PSL(n+1, \mathbb {C})$  we  say that  $\widetilde \gamma  \in GL(n+1,
\mathbb {C})$  is a  lift of $ \gamma  $ if there is an scalar $r\in
\mathbb {C}^*$ such that $r\widetilde \gamma  \in SL(n, \mathbb {C})$ and
$[r\widetilde  \gamma  ]_{n}= \gamma  $.

 Notice that
$ PSL(n+1, \mathbb {C})$  acts  transitively, effectively and by
biholomorphisms  on $\mathbb{P}^{n}_{\mathbb{C}}$, taking
   projective
subspaces  into projective subspaces.\\

In what follows $\C^{1,n}$ is a copy of $\C^{n+1}$ equipped with a
Hermitian form of signature $(1,n)$ that we assume is given by:
$$
<u,v>=-u_0\overline{v_0}+\sum_{j=1}^{n}u_j\overline{v_j} \,,
$$
where $u=(u_0,u_1,\ldots,u_n)$ and  $v=(v_0,v_1,\ldots,v_n)$. A
vector $v$ is called negative, null or positive depending (in the
obvious way) on the value of  $<v,v>$; we denote the set of
negative, null or positive vectors by $N_-, N_0$ and $N_+$
respectively. Thus one has:
$$N_- \,=\, \{u \in \C^{1,n} \, \large| \, u_0\overline{v_0} >
\sum_{j=1}^{n}u_j\overline{v_j}\} \, .
$$
The image $\mathbb {B}$ of $N_-$ in
$\P^n$ under the map $[\mbox{ }]_{n}$ is diffeomorphic to a
ball of real dimension $2n$, with boundary a sphere
$\mathbb {S}^{2n-1}$, which is the image of $N_0$.

If we let  $U(1,n) \subset GL(n+1, \C)$  be the subgroup  consisting of the elements  that preserve the
above Hermitian form, then its projectivization $[U(1,n)]_n$  is a subgroup of  $ PSL(n+1,
\mathbb {C})$ that we denote by $PU(1,n)$.

It is easy to see that $PU(1,n)$ acts transitively on $\mathbb {B}$
with isotropy $U(n)$. Let $\underline 0$ denote the center of
the ball $\mathbb {B}$, consider the space $T_{\underline 0} \mathbb {B}
\cong \C^n$ tangent to $\mathbb {B}$ at $\underline 0$, and put on it
the usual Hermitian metric on $\C^n$. Now we use the action of
$PU(1,n)$ to spread the metric, using that the action is
transitive and the isotropy is  $U(n)$, which preserves the usual
metric on $\C^n$. We thus get a Hermitian metric on $\mathbb {B}$,
which is clearly homogeneous. This metric is, up to scaling, the Bergman metric,
and gives the model we use for
complex hyperbolic $n$-space, that we denote by $\mathbb {H}_{\C}^n$.
It is clear from this construction that  $PU(1,n)$  is the group
of holomorphic isometries of $\mathbb {H}_{\C}^n$. Its elements  are classified as
follows  \cite{CG}:

\begin{definition}\label{d:classif}
The non-trivial elements of $PU(n, 1)$ fall into three general conjugacy types, depending on the
number and location of their fixed points:

\begin{enumerate}
\item Elliptic elements have a fixed point in $\mathbb {H}^n_\mathbb {C}$;

\item parabolic elements have a single fixed point on the boundary of $\mathbb {H}^n_\mathbb {C}$;
\item loxodromic elements have exactly two fixed points on the boundary of $\mathbb {H}^n_\mathbb {C}$;
\end{enumerate}
This exhausts all possibilities, see \cite{Goldman} for details.
\end{definition}

\begin{definition} Set  $\overline{\mathbb {H}}^n_\mathbb {C} =
  {\mathbb {H}}^n_\mathbb {C}  \cup \partial {\mathbb {H}}^n_\mathbb {C} $, and let $G$ be a discrete subgroup of $PU(1,n)$. The {\it
region of discontinuity} of $G$ in $\overline{\mathbb {H}}^n_\mathbb {C}$
is the set $\Omega = \Omega(G)$ of all points in
$\overline{\mathbb {H}}^n_\mathbb {C}$ which have a neighborhood that intersects only
finitely many copies of its $G$-orbit.
\end{definition}

Since $\mathbb {H}^n_\mathbb {C}$ with the Bergman metric is a connected,
metric space where each closed ball is compact and $PU(1,n)$ acts  by isometries,   the Arzel\`a-Ascoli theorem yields \cite{rat}:

\begin{theorem} \label{t:acc} Let $G$ be a subgroup of $PU(1,n)$. The following
three conditions are equivalent:

\begin{enumerate}

\item \label{i:1acc} The subgroup $G \subset PU(1,n)$ is discrete.

\item \label{i:2acc}The region of discontinuity of $G$  in
$\mathbb {H}^n_\mathbb {C}$ is all of $\mathbb {H}^n_\mathbb {C}$.

\item  \label{i:3acc} The region of discontinuity of $G$ in
$\mathbb {H}^n_\mathbb {C}$ is non-empty.
\end{enumerate}

\end{theorem}

 The following characterization of finite groups will be useful later, see \cite{CG, pablo}.

\begin{proposition} \label{c:clasfin}
Let $ G  \subset PU(1,n)$ be a discrete group, then  $ G  $ is finite if and only if
every element $ \gamma   \in  G  $ has finite order
 $o( \gamma  )$.
\end{proposition}

\section{Quasi-projective maps and equicontinuity}
Let us construct a ``completion" of the Lie group
$PSL(n,\mathbb {C})$ which is known as the space of quasi-projective
maps. These were introduced by Furstenberg and they are studied
in  \cite{AMS}.

 Let $\widetilde M:\mathbb {C}^{n+1}\rightarrow \mathbb {C}^{n+1}$ be a non-zero linear transformation which is not
 necessarily invertible. Let $Ker(\widetilde M)$ be its kernel and let $Ker(M)$ denote its projectivization. That is, $Ker(M) :=  [Ker(M)\setminus\{0\}]_n$ where $ [ \;]_n$ is the projection $\C^{n+1} \to \P^n$.
 Then the quasi-projective transformation induced by
 $\widetilde M$  is  the map $M:\mathbb {P}^{n}_\mathbb {C}\setminus Ker(M) \rightarrow
 \mathbb {P}^{n}_\mathbb {C}$ given by:
 $$
M([v])=[M(v)]_n\,;
 $$
this is well defined because $v\notin Ker(M)$. Moreover,
the  commutative diagram below implies that $M$
is a holomorphic map:
$$
 \xymatrix{
\mathbb {C}^{n+1}\setminus Ker (\widetilde M) \ar[rr]^{\widetilde M} \ar[d]_{[\mbox{ }]_n}&  & \mathbb {C}^{n+1}\setminus \{0\}\ar[d]^{[\mbox{ }]_n}\\
\mathbb {P}^{n}_\mathbb {C}\setminus Ker(M)
\ar[rr]_{M} && \mathbb {P}^{n}_\mathbb {C} } \;
$$
We denote by $QP(n,\mathbb {C})$ the space of all {\it quasi-projective maps} of
$\mathbb {P}^n_\mathbb {C}$. That is:
$$
 QP(n,\mathbb {C})= \{M = [\widetilde M]_n: \widetilde M \textrm{ is a non-zero linear
transformation of }
 \mathbb {C}^{n+1}\}
$$
Clearly  $PSL(n,\mathbb {C})\subset QP(n,\mathbb {C})$.

 A linear map $\widetilde
M:\mathbb {C}^{n+1}\rightarrow \mathbb {C}^{n+1}$ is said to be {\it a
lift
of} the quasi-projective map $M$ if $[\widetilde M]_n=M$. Conversely,
given a quasi-projective map  $M$ and
 a lift $\widetilde M$ we define
{\it the image of} $M$ as:
   $$Im(M)=[\widetilde {M}(\mathbb {C}^{n+1})\setminus\{0\}]_n\,.$$
Obviously we have $dim_\mathbb {C}(Ker(M))+dim_\mathbb {C}( Im(M))=n-1$.

\begin{proposition} \label{p:completes}
Let  $( \gamma_m)_{m\in \mathbb {N}}\subset PSL(n,\mathbb {C})$ be a
sequence of distinct elements, then  there  is a subsequence of
$( \gamma_m)_{m\in \mathbb {N}}$, still denoted $( \gamma_m)_{m\in
\mathbb {N}}$, and $ \gamma  \in QP(n,\mathbb {C})$ such that $ \gamma_m
\xymatrix{ \ar[r]_{m \rightarrow \infty}&}  \gamma  $ uniformly on
compact sets of $\mathbb {P}^n_\mathbb {C}\setminus Ker( \gamma  )$.

\end{proposition}

\begin{proof}  For each $m$, let $\tilde { \gamma  }_m=( \gamma_{ij}^{(m)})\in SL(n,\mathbb {C})$ be
a lift of $ \gamma_m$. Define  $$\vert
 \gamma_m\vert=\max\{\vert \gamma_{ij}^{(n)}\vert:i,j=1,3\}.$$
Notice that  $\vert  \gamma_m \vert$ is independent of the choice of lift in $SL(n,\C)$, and that
$\vert  \gamma_m \vert^{-1}\tilde{ \gamma  }_m$ is again a lift of
$ \gamma_m$. Since every bounded sequence in $\mathbb {C}$ has a
convergent subsequence we deduce that there is a subsequence of
$(\vert  \gamma_m \vert^{-1}\tilde{ \gamma  }_m)$, still denoted by
$(\vert  \gamma_m \vert^{-1}\tilde{ \gamma  }_m)$, and a non-zero $(n \times n)$-matrix  $
\tilde  \gamma  =( \gamma_{ij})\in Gl(n,\mathbb {C})$, such that:
$$
\vert  \gamma_m \vert^{-1} \gamma_{ij}^{(m)}\xymatrix{ \ar[r]_{m
\rightarrow \infty}&}  \gamma_{ij}.
$$
This implies that $\vert  \gamma_m
\vert^{-1}  \tilde  \gamma_{m}\xymatrix{ \ar[r]_{m \rightarrow \infty}&}
\tilde  \gamma  $ uniformly on compact sets of $\mathbb {C}^n$, regarded as linear transformations .

Now let $K\subset \P^n-Ker( \gamma  )$ be a compact set, so
$ \widehat{K}=\{\vert k\vert^{-1}k:[k]_n\in K\}$ is a compact set which
satisfies  $[ \widehat{K}]_n=K$. Thus
\begin{equation}\label{c:uni}
\vert  \gamma_m \vert^{-1}	\tilde \gamma_{m}\xymatrix{ \ar[r]_{m
\rightarrow \infty}&} \tilde  \gamma   \, \textrm{ uniformly on }  \widehat{K}.
\end{equation}
From this equation  and the facts $[\vert  \gamma_m
\vert^{-1} \gamma_m]_n= \gamma_m$ and $[ \widehat {K}]_n = K$, we
conclude that:
$$  \gamma_m
\xymatrix{ \ar[r]_{m \rightarrow \infty}&}  \gamma   \,
\textrm{ uniformly on } K\,,$$
where $ \gamma  =[\tilde  \gamma  ]_n$
\end{proof}
In what follows we will say that the sequence $( \gamma_m)\subset
PSL(n,\mathbb {C})$  to $ \gamma  \in QP(n,\mathbb {C})$ in the sense of quasi-projective transformations  if
$ \gamma_m \xymatrix{ \ar[r]_{m \rightarrow \infty}&}  \gamma  $
uniformly on compact sets of $\mathbb {P}^n_\mathbb {C}\setminus
Ker( \gamma  )$.

\medskip

We now recall:

\begin{definition}
The {\it equicontinuity region} for a family $G$ of
endomorphisms of  $\mathbb {P}^n_\mathbb {C}$, denoted $Eq(G)$, is
defined to be the set of points $z\in \mathbb{P}^n_\mathbb{C}$ for
which there is an open neighborhood $U$ of  $z$   such that $G
\vert_U$ is a normal family.  (Where normal family means that every
sequence of distinct elements has a subsequence which converges
uniformly on compact sets.)
\end{definition}

\begin{proposition} \label{p:disc}
 Let $( \gamma_m)\subset PSL(n,\mathbb {C})$ be a sequence which
converges to $ \gamma  \in QP(n,\mathbb {C})$, with  $Ker( \gamma  )$ being a
hyperplane. Let  $p\in Ker( \gamma  )\setminus Im( \gamma  )$,  let $U$ be a
neighborhood of $p$ and $\ell$   a line such that  $Ker( \gamma  )\cap \ell=\{p\}$. Then there is a
subsequence of $( \gamma_m)$, still denoted $( \gamma_m)$,  and a line
$\ell_p$, such that for every open   neighborhood  $W$ of $p$ with
compact closure in $U$, the set of cluster points of $\{
 \gamma_m(\overline W\cap \ell)\}$ is $\ell_p$.
 \end{proposition}

 \begin{proof}  Let $Gr_2(\mathbb {C}^{n+1})$ be the grassmannian of complex 2-planes in $\C^{n+1}$.
 Since  $Gr_2(\mathbb {C}^{n+1})$
is compact, there is as subsequence of $( \gamma_m)$, still denoted
$( \gamma_m)$, and a line $\ell_p$ such that
$ \gamma_m(\ell)\xymatrix{ \ar[r]_{m \rightarrow \infty}&}\ell_p$.
On the other hand, since the sequence $( \gamma_m)$ converges  to
$ \gamma  $ and $\ell\cap Ker( \gamma  )=\{p\}$ we conclude that
$\ell_p\cap Im( \gamma  )=Im( \gamma  )$ is a point, say $q$. Let $x
\in\ell_p\setminus \{q\}$, then there is a sequence $(y_m)\subset
\ell$ such that $(y_m)$ is convergent and $ \gamma_m(y_m)\xymatrix{
\ar[r]_{m \rightarrow \infty}&} x$. This implies that the limit
point of $(y_m)$ lies in $Ker( \gamma  )$, thence such a
point is $p$. In short $y_m \xymatrix{ \ar[r]_{m \rightarrow
\infty}&}p $ and $ \gamma_m(y_m)\xymatrix{ \ar[r]_{m \rightarrow
\infty}&} x$.
 \end{proof}

As an inmediate consequence one has:

\begin{corollary} Let
$( \gamma_m)\subset PSL(n,\mathbb {C})$ be a sequence which converges
to $ \gamma  \in QP(n,\mathbb {C})$. If $Ker( \gamma  )$ is a hyperplane,
then the equicontinuity set is:
$$Eq(\{  \gamma_m:m\in \mathbb {N}\})=\mathbb {P}^n_\mathbb {C}\setminus
 Ker( \gamma  ).$$
 \end{corollary}

\section{The limit set according to Chen and Greenberg}

The main result of  this section is Lemma \ref{l:conv}, which
 is useful for proving  properties about
subgroups of $PU(1,n)$. This lemma is used in the following   Section \ref {sec:main-thm} for proving Theorem 1. We use  \ref{l:conv}  also in this section,
 to give direct proofs of several important results from
  \cite{CG} used in the sequel.

\begin{definition} Let $G$ be a discrete subgroup of $\rm{Iso}(\mathbb {H}^n)$. The {\it limit set}
of $G$ in the sense of Chen-Greenberg, denoted $\Lambda_{CG}(G)$
or simply $\Lambda_{CG}$, is the set of accumulation points in
$\overline{\mathbb {H}}^n_\mathbb {C}$ of orbits of points in  $
\mathbb {H}^n_\mathbb {C}$.
\end{definition}

\begin{lemma} \label{l:conv}
Let $G\subset PU(1,n)$ be a discrete group, $( \gamma_m)_{m\in
\mathbb {N}}\subset G$  a sequence of distinct elements and
$ \gamma  \in QP(n,\mathbb {C})$ such that $( \gamma_m)$ converges to
$ \gamma  $ in the sense of quasi-projective transformations, then:
\begin{enumerate}
\item \label{i:conv1} The image $Im( \gamma  )$ is a point in $\partial
\mathbb {H}^n_\mathbb {C}$.

\item \label{i:conv2} The kernel $Ker( \gamma  )$ is a hyperplane tangent to
$\partial \mathbb {H}^n_\mathbb {C}$.

 \item \label{i:conv3} One has
$Ker( \gamma  )\cap \partial \mathbb {H}^n_\mathbb {C}\in
\Lambda_{CG}(G)$.
\end{enumerate}
\end{lemma}

\begin{proof} Let us prove by contradiction
 (\ref{i:conv1}). Since $\gamma$ is holomorphic the set  $ \gamma  (
\mathbb {H}^n_\mathbb {C}\setminus Ker( \gamma  ))$ is an open set  in the projective subspace $Im(\gamma)$. On the other hand,
by  \ref{t:acc} the set $ \gamma  (
\mathbb {H}^n_\mathbb {C}\setminus Ker( \gamma  ))$ is  contained
in $Im( \gamma  )\cap\partial\mathbb {H}^n_\mathbb {C}$, which has empty interior ( whenever  $Im( \gamma  )$ is not a point), which is  a contradiction.  In the rest of the proof the unique element in $Im( \gamma  )$ 	
will be denoted by $q$.

Now let us prove (\ref{i:conv2}). Assume that $Ker( \gamma  )\cap
\mathbb {H}^n$ is not empty. Thus we can choose $p\in
Ker( \gamma  )\setminus Im( \gamma  )\cap \mathbb {H}^n$. Applying
Proposition \ref{p:disc} to $(\gamma_m), \, \gamma  , \,p, \,
\mathbb {H}^n_\mathbb {C}$ it it followed that there is  line  $\ell_p$
contained in $\mathbb {H}^n_\mathbb {C}$, which is a contradiction, since $\Bbb{H}^n_\Bbb{C}$ does not contains complex lines. Therefore  $Ker( \gamma  )\cap \mathbb {H}^n=\emptyset$.

Assume now that
$Ker( \gamma  )\cap\overline{\mathbb H}^n_\mathbb {C}=\emptyset$. From this and (\ref{i:conv1}) of the present lemma we conclude  that
$ \gamma_m$ converges uniformly to the constant function $q$ on
$\overline{ \mathbb {H}}^n_\mathbb {C}$. Let $x\in \mathbb {H}^n_\mathbb {C}$ and
$U$ be a neighborhood of $p$ such that
		$U\cap\mathbb {H}^n_\mathbb {C}\subset\mathbb {H}^n_\mathbb {C}\setminus\{x\}$.
The uniform convergence  implies  that  there is  a
natural number $n_0$ such that
$ \gamma_m(\overline{\mathbb H}^n_\mathbb {C})\subset U\cap
\mathbb {H}^n_\mathbb {C}\subset\mathbb {H}^n_\mathbb {C}\setminus\{x\}$ for each
$m>n_0$. This is a contradiction since each $ \gamma_m$ is a
homeomorphism.

Let us prove (\ref{i:conv3}). By \ref{p:completes} we can assume
that there is $\tau\in QP(n\c)$ such that $( \gamma  ^{-1}_m)$
converges to $\tau$ in the sense of quasi-projective transformations. Thus by (\ref{i:conv1}) of the present lemma
we have that $Im(\tau)$ is a point $p$ in $\Lambda_{Kul}( G  )$. We
claim that $\{p\}=Im(\tau)=Ker( \gamma  )\cap \partial
\mathbb {H}^n_\mathbb {C}$. Assume  this  does not
happen; let $x \in \mathbb {H}^n_\mathbb {C}$, then
$ \gamma  ^{-1}_m(x)\xymatrix{ \ar[r]_{m \rightarrow \infty}&} p$,
thus $\{ \gamma  ^{-1}_m(x):m\in \mathbb {N}\}\cup \{ p\}$ is a compact
set which lies in $\mathbb {P}^n_\mathbb {C}\setminus Ker( \gamma  )$, thus  $x= \gamma_m( \gamma  ^{1}_m(x))\xymatrix{ \ar[r]_{m \rightarrow
\infty}&} q$.
Which is a contradiction.

\end{proof}

From the proof of the previous result one gets:

\begin{corollary} \label{c:inversas}
Let $G\subset PU(1,n)$ be a discrete group, $( \gamma_m)_{m\in
\mathbb {N}}\subset G$  a sequence of distinct elements and
$ \gamma  \in QP(n\mathbb {C})$ such that $( \gamma_m)$ converges to
$ \gamma  $ in the sense of quasi-projective transformations. Then there are a subsequence of $(\gamma_m)$, still denoted $(\gamma_m)$, and an element  $ \tau  \in QP(n,\mathbb {C})$ such that:

\begin{enumerate}
 \item  The sequence $( \gamma_m^{-1})$ converges to $ \tau  $ in the sense of quasi-projective transformations;

\item The image $Im(\tau)$ is a point in $\partial
\mathbb {H}^n_\mathbb {C}$ and  $Ker( \tau)$ is a hyperplane tangent to
$\partial \mathbb {H}^n_\mathbb {C}$.

\item   One has $Im(\tau)=Ker(\gamma)\cap\partial\mathbb {H}^n_\mathbb {C}$;

\item  Also $Im(\gamma)=Ker(\tau)\cap\partial\mathbb {H}^n_\mathbb {C}$.
\end{enumerate}
\end{corollary}

 \begin{theorem}\label{t.lim-indep}
Let $G$ be as above. Then the limit set $\Lambda_{CG}(G)$ is
independent of the choice of orbit.
\end{theorem}

\begin{proof} Let $x,y\in \mathbb {H}^n_\mathbb {C}$, and let $z$ be a cluster point of $Gy$.
Then there exists a sequence $(g_m)\subset G$  such that $g_m(y)$
converges to $z$. It  follows from lemma \ref{l:conv}  that  $z$  is
also a cluster point of $(g_m(x))$, which ends the proof.
\end{proof}

It is clear from the definitions that the limit set
$\Lambda_{CG}(G)$ is a closed, $G$-invariant set, and it is empty
if and only if $G$ is finite (since every sequence in a compact
set contains convergent subsequences). Moreover one has:

\begin{theorem}  Let $G$ be a discrete group such that   $\Lambda_{CG}(G)$  has more than two points, then it has infinitely many points.

\end{theorem}

\begin{proof}
Assume that $\Lambda_{CG}(G)$
 is finite with at least 3 points. Then $$\tilde G=\bigcap_{x\in
 \Lambda_{CG}(G)}Isot(x,G)$$ is a normal subgroup of $G$ with finite
 index.  Moreover, by \ref{d:classif} each element in $PU(1,n)$ has at
 most 2 fixed points in $\partial \mathbb {H}^n_\mathbb {C}$. Hence $\tilde
 G$ is trivial and therefore $G$ is finite, which is a contradiction.
 \end{proof}

\begin{definition}
The group $G$ is elementary if $\Lambda_{CG}(G)$ has at most two
points.
\end{definition}

\begin{theorem} \label{t:minimal} If $G\subset PU(1,n)$ is a non-elementary discrete group, then
 $\Lambda_{CG}(G)$ is the unique closed minimal set.
\end{theorem}

\begin{proof}
Let $S$ be a closed invariant set and  $z\in S$. By \ref{l:conv} we know  there is an acumulation point $\tilde z\in S$  of $Gz$ such that $\tilde z\in \Lambda_{CG}(G)$.  Let $y\in \Lambda_{CG}(G)$,
 then there is a sequence $(g_m)\subset G
$ and a point $p\in \mathbb {H}^n_\mathbb {C}$ such that $g_m(p)$
converges to $y $. By Lemma \ref{l:conv} there are points $p,q\in
\partial \mathbb {H}^n_\mathbb {C}$ such that we can assume that $g_m$
converges uniformly to $y$ in compact sets of
$\overline{\mathbb H}_\mathbb {C}^n\setminus \{p\}$. Now, it is well know (see Theorem 3.1  in
 \cite{kamiya})  that there is a transformation $g\in G$ such that
$g(p)\neq p$. Hence  we can assume that $\tilde z\neq q$,  so
 we conclude that $g_m(\tilde z)$ converges to $y$.

\end{proof}

\begin{remark} Notice that the previous result implies that if
$G$  non-elementary, then $\Lambda_{CG}(G)$ is
a nowhere dense perfect set.
In other words  $\Lambda_{CG}(G)$
has empty interior and every orbit in $\Lambda_{CG}(G)$  is dense
in $\Lambda_{CG}(G)$.
\end{remark}

\section{On the equicontinuity region}\label{sec:main-thm}

The sphere $ \partial \mathbb {H}^n_\mathbb {C}$ has real codimension 1 in
$\mathbb  P^n_{\mathbb  C}$  and its tangent bundle   contains a maximal  subbundle  which is a complex subbundle of
$T\mathbb  P^n_{\mathbb  C}\vert_{\partial \mathbb {H}^n_\mathbb {C}}$.  In fact
 this subbundle  defines the canonical contact structure on the sphere. We let
$\mathcal{C}(\partial \mathbb {H}^n_\mathbb {C} )$ be the union of all complex projective hyperplanes tangent to $ \partial \mathbb {H}^n_\mathbb {C}$.
Given a  discrete subgroup $G\subset PU(1,n)$ set:
$$\mathcal{C}(G):=  \mathcal{C}(\partial \mathbb {H}^n_\mathbb {C}) \vert_{\Lambda_{CG}(G)}   =  \bigcup_{p\in \Lambda_{CG}(G)}\mathcal{H}_p \,,$$
where $\mathcal{H}_p$ denotes the hyperplane tangent to $\partial
\mathbb {H}^n_\mathbb  C$ at a point  $p \in \Lambda_{CG}(G)$.
It is clear that $\mathcal{C}(G)$ is a closed $G$-invariant subset of $\mathbb  P^n_{\mathbb  C}$.

The following lemma proves part of Theorem \ref{Thm:main}.

\begin{lemma}
The equicontinuity region of $G$ is:
$$Eq(G) = \mathbb  P^n_{\mathbb  C} \setminus \mathcal{C}(G)\,.$$

\end{lemma}
\begin{proof} Since $G$ is infinite and discrete, it contains at least a parabolic or a loxodromic element $\gamma$. Let  $x_0$  be a  fixed point of $\gamma$. By  Corollary \ref{c:inversas}  we can ensure that the hyperplane $\mathcal {H}_{x_0}$
tangent to $\partial \mathbb {H}^n$
 at  $x_0$ is contained in $\mathbb {P}^n_\mathbb {C}\setminus Eq( G  )$. On th other hand by Theorem \ref{t:minimal}, the closure of the orbit of $x_0$ is
  $\Lambda_{CG}(G)$.
  Thus $Eq( G  )\subset \mathbb {P}^{n}_\mathbb {C}\setminus\mathcal{C}( G  )$.
Let us now prove $ \mathbb {P}^{n}_\mathbb {C}\setminus\mathcal{C}( G  ) \subset Eq( G  )$.
Let $p\in \mathbb {P}^{n}_\mathbb {C}   \setminus\mathcal{C}( G  )$ and
$( \gamma_m)_{m\in \mathbb {N}}$ a sequence of distinct elements. By
Lemma  \ref{l:conv}  there  are points $p,q\in
\Lambda_{CG}( G  )$ and  a subsequence of $( \gamma_m)$, still
denoted $( \gamma_m)_{m\in \mathbb {N}}$, such that $ \gamma_m
\xymatrix{ \ar[r]_{m \rightarrow \infty}&} q$ uniformly on compact
sets of $\mathbb {P}^n_\mathbb {C}\setminus \mathcal{H}_p$, where $\mathcal{H}_p$ denotes the
hyperplane tangent  to $\partial \mathbb {H}^n$  at $p$.
This completes the proof.
\end{proof}

The following result completes the proof of Theorem \ref{Thm:main}.

\begin{corollary} \label{c:orbit}
Let $G\subset PU(1,n)$ be a discrete subgroup. Then
 $G$ acts discontinuously on $Eq(G)$ and, moreover, for every compact set $K\subset Eq( G  )$ the cluster
points of the orbit $ G   K$ lie in $\Lambda_{CG}( G  )$.

\end{corollary}

\begin{proof}
Assume on the contrary  that $G$ does not act discontinuously on $Eq(\Gamma)$. Then there is a compact set $K$ and a sequense of distinct elements $(\gamma_m)\subset G$, such that $\gamma_m(K)\cap K\neq \emptyset$. By Proposition \ref{p:completes}, there is a subsequence of  $(\gamma_m)$, still
denoted $(\gamma_m)$, and $\gamma\in QP(n,\mathbb {C})$, such that $(\gamma_m)$ converges to $\gamma$ in the sense of quasi-projective transformations. Moreover, by Lemma \ref{l:conv}, $Im(\gamma)$ is a
point $p$ in $\partial \mathbb {H}^n_\mathbb {C}$ and $Ker(\gamma)$ is a hyperplane tangent to $\partial \mathbb {H}^n_\mathbb {K}$. Therefore there is a neighborhood
$U$ of $p$ disjoint from $K$ and a natural number $n_0$ such that
$\gamma_m(K)\subset U$ for all $m>n_0$. This implies $\gamma_m(K)\cap K=\emptyset$, which is a contradition. Therefore $\Gamma$ acts discontinuously on $Eq(G)$.

From the previous argument we deduce also that for every compact set $K\subset Eq( G  )$ the cluster
points of $ G   K$ lie in $\Lambda_{CG}( G  )$.
\end{proof}

\bibliographystyle{amsplain}

\end{document}